\documentclass{article}
\usepackage{amsmath, amssymb, graphicx}
\usepackage[margin=1in]{geometry} 
\usepackage{amsthm} 
\usepackage{booktabs} 
\usepackage{multirow} 

\usepackage{microtype}

\usepackage{hyperref} 

\newtheorem{theorem}{Theorem}[section] 
\newtheorem{proposition}[theorem]{Proposition} 
\newtheorem{definition}[theorem]{Definition} 
\theoremstyle{remark} 
\newtheorem*{remark}{Remark} 

\title{The Resonance Bias Framework: Resonance, Structure, and Arithmetic in Quadrature Error}
\author{William Cook\thanks{Second-year BSc Economics student. This paper and several others were written within a few months as part of an experimental independent research program exploring the use of AI tools (e.g., Gemini Pro 2.5, 4o, and o4-mini-high) to accelerate paper writing while maintaining mathematical rigor.}}

\date{April 22, 2025} 

\begin{document}

\maketitle

\begin{abstract}
While classical analysis often describes quadrature errors for periodic functions via asymptotic rates or the aliasing sum ($B_P = \sum_{l\neq 0} c_{lP}$), we argue that for the trapezoidal rule on uniform grids, the error possesses a rich, deterministic structure. Analysing the prototype function $f(x) = \sin^2(2\pi k x)$, we derive an analytical error formula governed by a complex resonance function $\tilde{\chi}_P(y)$. This function, connected to the Dirichlet kernel, roots of unity, and discrete Fourier analysis on $\mathbb{Z}/P\mathbb{Z}$, acts as a spectral filter, linking the error to the interaction between the function's frequency content and the grid structure via arithmetic properties (such as $k/P$) and geometric phase cancellation, visualised through vector averaging on the unit circle. Our Resonance Bias Framework (RBF) generalises this to arbitrary smooth periodic functions, yielding the expression $B_P[f] = \sum_{k\neq 0} c_k \tilde{\chi}_P(k/P)$. While mathematically equivalent to the classical aliasing sum, this formulation reveals the underlying mechanism: quadrature error arises from a structured resonance phenomenon in which the grid filter $\tilde{\chi}_P$ selectively modulates the spectrum. The RBF thus provides a perspective on numerical integration error not as mere aliasing noise, but as a structured, interpretable outcome governed by number theory and geometric structure—offering insight into the error structure at finite resolution beyond traditional asymptotic or statistical approaches.
\end{abstract}


\section{Introduction}

\subsection{Context}
Numerical integration, or quadrature, is a fundamental tool across science and engineering. Among the simplest methods, the trapezoidal rule is widely used, particularly for periodic functions where it is known to exhibit surprisingly high accuracy.

\subsection{Classical View}
Standard error analysis, often using the Euler-Maclaurin formula or Fourier analysis \cite{Trefethen2013, SteinShakarchi2003}, attributes the trapezoidal rule error for periodic functions primarily to the function's smoothness, or more specifically, to the phenomenon of aliasing of Fourier coefficients. The error \(B_P[f]\) for a $P$-point rule is typically expressed as
\[
B_P[f] = \sum_{l \in \mathbb{Z} \setminus \{0\}} c_{lP},
\]
where $c_k$ are the Fourier coefficients of the integrand and only terms whose indices are multiples of $P$ contribute. Errors are often viewed in aggregate as bounded noise or analysed asymptotically based on convergence rates (e.g., $O(P^{-s})$ for $C^s$ functions, or exponential decay $O(e^{-\gamma P})$ for analytic functions \cite{TrefethenWeideman2014}).

\subsection{Motivation}
While these classical results provide essential convergence rates and bounds, they often obscure the detailed, deterministic structure inherent in the error itself, especially for a finite number of grid points $P$. By focusing on asymptotic behavior or treating the error statistically, classical analysis can obscure the precise, non-asymptotic structure present for any given $P$. We propose that for periodic functions on uniform grids, this error is not random noise but a structured signal possessing distinct algebraic, geometric, and number-theoretic properties. The "Resonance Bias Framework" (RBF) presented here aims to reveal, analyse, and leverage this underlying structure.\footnote{We adopt the name Resonance Bias Framework (RBF) to emphasise the grounding of the framework in the frequency-grid interaction quantified by the resonance function, moving away from earlier conceptualisations that might have implied broader universality than currently demonstrated.}

\subsection{Approach}
Our analysis begins with an analytically tractable prototype function, $f(x) = \sin^2(2\pi k x)$. Direct calculation of the $P$-point trapezoidal sum for this function reveals the emergence of a key complex resonance function, $\tilde{\chi}_P(y)$, which precisely encapsulates the interaction between the function's frequency content and the discretisation grid.

\subsection{Contribution}
We define and analyse the complex resonance function $\tilde{\chi}_P(y)$ and its real part $\chi_P(y)$, elucidating their fundamental properties and interpretations—geometric, algebraic, and number-theoretic. We show that $\tilde{\chi}_P(y)$ acts as a spectral filter, quantifying the grid's response to input frequencies. We then generalise the bias formula for arbitrary smooth periodic functions (specifically, those whose Fourier series converge absolutely), demonstrating that the RBF bias formulation,
\[
B_P[f] = \sum_{k \neq 0} c_k \tilde{\chi}_P(k/P),
\]
is mathematically equivalent to the classical aliasing sum ($\sum_{l \neq 0} c_{lP}$) but provides mechanistic insight into how the error arises from grid resonance and structured phase cancellation.

\subsection{Outline}
Section 2 briefly reviews the classical Fourier error framework based on aliasing. Section 3 develops the RBF by analysing the $\sin^2$ prototype function, leading to the definition of the resonance function $\tilde{\chi}_P$. Section 4 delves into the mathematical properties and interpretations (geometric, algebraic, number-theoretic) of $\tilde{\chi}_P(y)$. Section 5 presents the generalised RBF error formula, proves its equivalence to the classical result, and extends the framework to two dimensions. Section 6 discusses the broader implications of the RBF perspective. Section 7 outlines future research directions and provides concluding remarks.

\section{The Classical Fourier Error Analysis Framework (Brief Review)}

Classical analysis of the trapezoidal rule for a 1-periodic function $f(x)$ often starts with its Fourier series representation:
\[
f(x) = \sum_{k=-\infty}^{\infty} c_k e^{2\pi i k x}
\]
where $c_k = \int_0^1 f(x) e^{-2\pi i k x} dx$. The exact integral over $[0,1]$ is $I[f] = c_0$.

The $P$-point composite trapezoidal rule approximation is given by:
\[
I_P[f] = \frac{1}{P} \sum_{j=0}^{P-1} f\left(\frac{j}{P}\right)
\]
Using Poisson summation or direct substitution of the Fourier series into the sum and leveraging properties of roots of unity (specifically, that $\sum_{j=0}^{P-1} (e^{2\pi i k/P})^j = P$ if $k \equiv 0 \pmod P$ and 0 otherwise), the error \(B_P[f] = I_P[f] - I[f]\) can be shown \cite{Trefethen2013, SteinShakarchi2003, TrefethenWeideman2014} to be:
\begin{equation} \label{eq:classical_error}
B_P[f] = \sum_{l \in \mathbb{Z} \setminus \{0\}} c_{lP}
\end{equation}
This formula is typically interpreted as aliasing: the discrete sum $I_P[f]$ equals the sum of all Fourier coefficients whose indices are multiples of $P$. The error arises because the discrete sampling cannot distinguish frequencies $k$ from $k+lP$. The rapid decay of $c_k$ for smooth functions explains the high accuracy and convergence rates ($O(P^{-s})$ or $O(e^{-\gamma P})$).

\section{The Resonance Bias Framework: Analysis of a Prototype Function}

The Resonance Bias Framework (RBF) posits that numerical integration errors, particularly for periodic functions evaluated on uniform grids using rules like the trapezoidal rule, possess a deterministic structure rooted in algebraic and geometric principles. To elucidate this framework, we analyse the trapezoidal rule's bias for the prototype function $f(x)=\sin^2(2\pi kx)$, where $k \in \mathbb{R}^+$ defines the frequency and $x \in [0,1]$. This choice, characterised by a single dominant frequency component, permits an exact, analytical derivation of the bias, revealing the emergence of a fundamental \textit{resonance function} that governs the error's behavior. This approach contrasts with classical error bounds, such as those derived from the Euler-Maclaurin formula or based solely on function smoothness (e.g., \cite{Henrici1982}), which often provide asymptotic decay rates (e.g., $O(P^{-2})$ for $f \in C^2$) but obscure the precise, often oscillatory, error structure present at a finite number of grid points $P$. Our analysis unveils this underlying structure, connecting it directly to the arithmetic interaction between the function's frequency content and the discretisation grid.

\paragraph{3.1 The Prototype Function and Quadrature Setup}
We consider the function $f(x)=\sin^2(2\pi kx)$, where $k>0$ is the frequency parameter. We extend $f(x)$ periodically from $[0,1]$ to $\mathbb{R}$. For analytical convenience, let $m=2k$, where \(m\) can be any positive real number, representing twice the fundamental frequency \(k\). The function can then be expressed in terms of this effective frequency $m$:
\[
f(x) = \frac{1}{2} - \frac{1}{2} \cos(2\pi m x)
\]
We assume $m>0$. Our objective is to analyse the bias $B_P[f]$ of the $P$-point composite trapezoidal rule ($P \in \mathbb{Z}^+, P \ge 2$) when approximating the integral $I[f]=\int_0^1 f(x) dx$. The trapezoidal rule approximation is given by the discrete average over the uniform grid $x_j = j/P$ for $j=0,1,\dots,P-1$:
\[
I_P[f] = \frac{1}{P} \sum_{j=0}^{P-1} f\left(\frac{j}{P}\right)
\]
The bias, or quadrature error, is defined as the difference between the approximation and the exact integral:
\[
B_P[f] = I_P[f] - I[f]
\]

\paragraph{3.2 Analytical Integral Evaluation}
The exact integral $I[f]$ is computed directly from the cosine form of $f(x)$:
\[
I[f] = \int_0^1 \left( \frac{1}{2} - \frac{1}{2} \cos(2\pi m x) \right) dx = \frac{1}{2} - \frac{1}{2} \int_0^1 \cos(2\pi m x) \, dx
\]
Since $m>0$, the cosine integral evaluates to:
\[
\int_0^1 \cos(2\pi m x) \, dx = \left[ \frac{\sin(2\pi m x)}{2\pi m} \right]_0^1 = \frac{\sin(2\pi m)}{2\pi m}
\]
Thus, the exact integral is:
\[
I[f] = \frac{1}{2} - \frac{\sin(2\pi m)}{4\pi m}
\]
We introduce the continuous correction term $C(m) := \frac{\sin(2\pi m)}{4\pi m}$. Note that \(C(m) = 0\) if and only if \(m\) is a non-zero integer, reflecting the exact integration of integer-frequency cosines over \([0,1]\). This term represents the deviation of the exact integral from the function's mean value of $1/2$. For non-integer $m$, $C(m)$ is an oscillatory term decaying as $O(m^{-1})$, bounded by $|C(m)| \le \frac{1}{4\pi m}$. In the limit $m \to 0$, L'Hôpital's rule gives $\lim_{m \to 0} C(m)=1/2$. Using this definition, the exact integral can be expressed concisely as:
\[
I[f] = \frac{1}{2} - C(m)
\]

\paragraph{3.3 Discrete Approximation via Trapezoidal Rule}
Applying the $P$-point trapezoidal rule to $f(x)$ yields:
\[
I_P[f] = \frac{1}{P} \sum_{j=0}^{P-1} \sin^2\left(2\pi k \frac{j}{P}\right) = \frac{1}{P} \sum_{j=0}^{P-1} \left( \frac{1}{2} - \frac{1}{2} \cos\left(2\pi m \frac{j}{P}\right) \right)
\]
Separating the constant term from the finite cosine sum gives:
\[
I_P[f] = \frac{1}{P} \left( \sum_{j=0}^{P-1} \frac{1}{2} \right) - \frac{1}{2P} \sum_{j=0}^{P-1} \cos\left(2\pi \frac{m j}{P}\right)
\]
This simplifies to:
\[
I_P[f] = \frac{1}{2} - \frac{1}{2P} S_P(m)
\]
where $S_P(m)$ denotes the finite cosine sum central to the discrete approximation:
\[
S_P(m) := \sum_{j=0}^{P-1} \cos\left(2\pi \frac{m j}{P}\right)
\]
The behavior of the trapezoidal sum hinges entirely on the structure of $S_P(m)$.

\paragraph{3.4 Emergence of the Resonance Function}
To analyse the structure of $S_P(m)$, we introduce the complex resonance function $\tilde{\chi}_P(y)$ and its real part $\chi_P(y)$, formally defined in Section 4 (Definition \ref{def:resonance_functions}). For now, we use the definition:
\[
\chi_P(y) := \frac{1}{P} \sum_{j=0}^{P-1} \cos(2\pi y j)
\]
The real resonance function $\chi_P(y)$ quantifies the degree of constructive or destructive interference arising when sampling a cosine wave of relative frequency $y$ (relative to the sampling frequency $P$) on the $P$-point uniform grid.
Using this definition with the relative frequency $y=m/P$, the sum $S_P(m)$ is directly proportional to $\chi_P(m/P)$:
\[
S_P(m) = P \cdot \chi_P\left(\frac{m}{P}\right)
\]
Substituting this relation back into the expression for the discrete sum $I_P[f]$ provides a remarkably simple form:
\[
I_P[f] = \frac{1}{2} - \frac{1}{2P} \left( P \cdot \chi_P\left(\frac{m}{P}\right) \right) = \frac{1}{2} - \frac{1}{2} \chi_P\left(\frac{m}{P}\right)
\]
This result directly links the discrete approximation to the resonance function evaluated at the relative frequency $m/P$.

\paragraph{3.5 The Deterministic Bias Theorem for $\mathbf{\sin^2(2\pi kx)}$}
Having derived analytical expressions for both the continuous integral $I[f]$ and the discrete sum $I_P[f]$, we arrive at the exact bias formula for the prototype function.

\begin{theorem}[Deterministic Bias for $\mathbf{\sin^2}$] \label{thm:bias_sin2}
Let $f(x)=\sin^2(2\pi kx)$ with $m=2k>0$. For the $P$-point composite trapezoidal rule ($P \ge 2$), the bias $B_P[f]=I_P[f]-I[f]$ is given exactly by:
\begin{equation} \label{eq:bias_sin2}
B_P[f] = -\frac{1}{2} \chi_P\left(\frac{m}{P}\right) + C(m)
\end{equation}
where $\chi_P(y)$ is the real resonance function (Definition \ref{def:resonance_functions}) and $C(m)=\frac{\sin(2\pi m)}{4\pi m}$.
\end{theorem}

\begin{proof}
The bias is obtained by direct subtraction using the derived expressions for $I_P[f]$ and $I[f]$:
\[
B_P[f] = \left( \frac{1}{2} - \frac{1}{2} \chi_P\left(\frac{m}{P}\right) \right) - \left( \frac{1}{2} - C(m) \right)
\]
Algebraic simplification immediately yields:
\[
B_P[f] = -\frac{1}{2} \chi_P\left(\frac{m}{P}\right) + C(m)
\]
This completes the proof.
\end{proof}

\begin{remark}
If \(m/P\) is an integer \(n\), then \(\chi_P(m/P) = \chi_P(n) = 1\) (by Proposition \ref{prop:chi_properties}). In this resonant case, the bias becomes \(B_P[f] = -1/2 + C(m)\). If \(m\) itself is also an integer, \(C(m)=0\), yielding a bias of exactly \(-1/2\).
\end{remark}

Theorem \ref{thm:bias_sin2} forms a cornerstone of the Resonance Bias Framework. It demonstrates that the bias for this fundamental periodic function is not merely statistical noise or a quantity described only by its asymptotic decay rate, but rather a precisely determined value. The bias is primarily controlled by the resonance term $\chi_P(m/P)$, which oscillates between $-1$ and $1$ as a function of the arithmetic relationship between the frequency $m$ and the grid size $P$. The term $C(m)$ represents a frequency-dependent offset independent of the grid resolution $P$.

This analytical formula \eqref{eq:bias_sin2} provides a structurally-focused perspective compared to classical Fourier error analysis for general periodic functions. While approaches like \cite{Trefethen2013} express the error as an infinite sum over aliased Fourier coefficients (\(B_P = \sum_{l \neq 0} c_{lP}\), see \eqref{eq:classical_error}), the RBF result \eqref{eq:bias_sin2} isolates the contribution of the function's primary frequency content (via $m$) through the action of the interpretable filter $\chi_P(m/P)$. The resonance function $\chi_P(m/P)$ quantifies the extent to which the frequency $m$ constructively or destructively interferes with the grid structure, effectively measuring how $m$ aliases onto the grid's natural frequencies. This connection, formalised in Section 5, highlights how the RBF interprets the phenomenon of aliasing as a structured resonance effect governed by $\chi_P$.

\paragraph{3.6 Resonance Landscape and Comparison with Classical Bounds}
The exact structure of the bias, as dictated by the resonance function $\chi_P(m/P)$, can be visualised as a 'resonance landscape' plotted against the relative frequency $y=m/P$, as shown in Figure \ref{fig:chi_P_landscape}. This landscape exhibits sharp peaks of height 1 at integer values $y=0,1,2,\dots$, corresponding to full resonance where the frequency $m$ is an integer multiple of $P$. Between these peaks, the function oscillates rapidly, exhibiting exact zeros at rational values $y=n/P$ for integers $n$ such that $n \not\equiv 0 \pmod{P}$. These zeros correspond to perfect phase cancellation among the terms in the defining sum of $\chi_P$. The plot visually demonstrates how the quadrature bias $B_P[f]$, dominated by the term $-\frac{1}{2}\chi_P(m/P)$, can be determined in magnitude and sign by locating the value $m/P$ (indicated by the star in Figure \ref{fig:chi_P_landscape}) within this intricate, structured landscape.

\begin{figure}[htbp]
\centering
\includegraphics[width=0.8\textwidth]{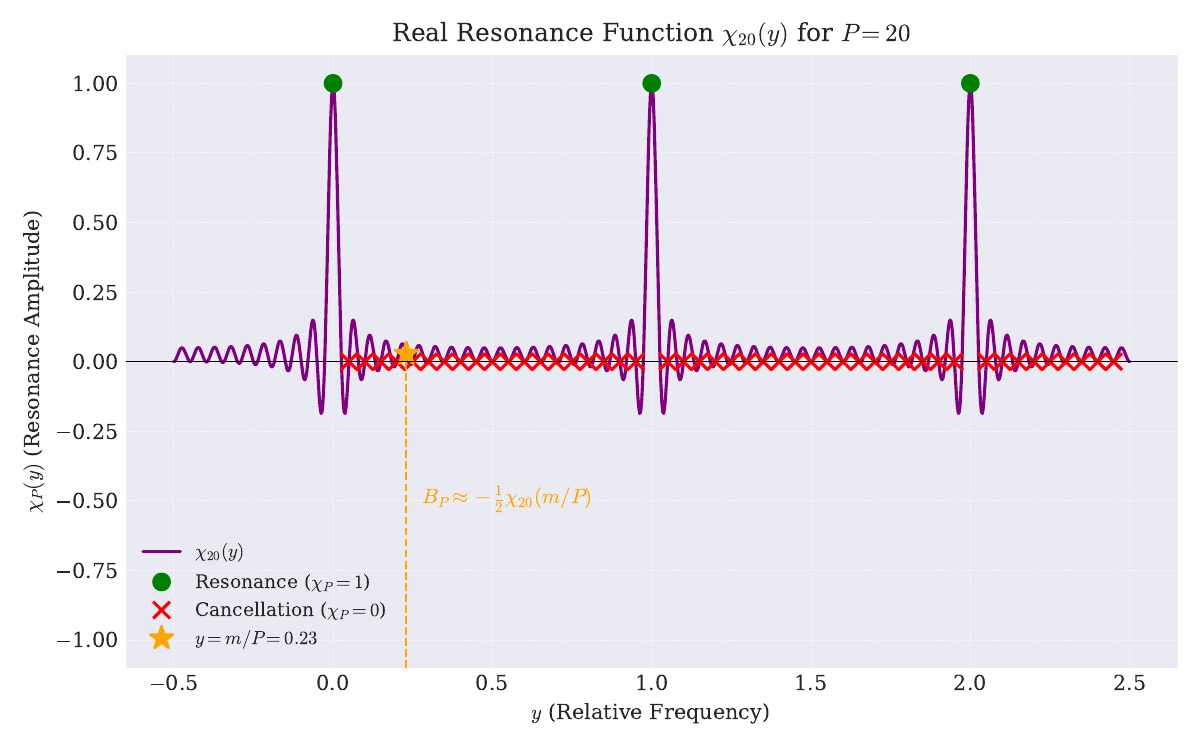} 
\caption{The Real Resonance Function $\chi_P(y)$ for $P=20$, illustrating the resonance landscape. The function $\chi_{20}(y)$ (purple line) quantifies the grid's response. Peaks of height 1 (green circles) indicate full resonance at integer relative frequencies $y=0, 1, 2$. Zeros (red crosses) indicate perfect cancellation at $y=n/P$ where $n \not\equiv 0 \pmod{P}$. The orange star marks $y=m/P=0.23$ corresponding to the prototype function $f(x)=\sin^2(2\pi(2.3)x)$. Theorem \ref{thm:bias_sin2} implies the bias $B_P[f] \approx -\frac{1}{2}\chi_{20}(0.23)$, indicated qualitatively by the dashed orange line.}
\label{fig:chi_P_landscape}
\end{figure}

This detailed view contrasts with standard error bounds. For a smooth periodic function $f \in C^2[0,1]$, classical analysis provides the bound $|B_P[f]| \le \frac{1}{12P^2} \|f''\|_\infty$ \cite{Trefethen2013}. For our prototype $f(x)=\sin^2(2\pi kx)$, we have $f''(x) = -(2\pi m)^2 \cos(2\pi mx)$, giving $\|f''\|_\infty = (2\pi m)^2 = 4\pi^2 m^2$. The classical bound thus becomes:
\[
|B_P[f]| \le \frac{4\pi^2 m^2}{12 P^2} = \frac{\pi^2 m^2}{3 P^2}
\]
While this bound holds (since \(f \in C^\infty\)) and correctly captures the overall $O(P^{-2})$ decay with respect to the grid size $P$, it averages out the detailed error structure revealed by $\chi_P(m/P)$, which depends critically on the arithmetic nature of $m/P$. The analytical formula \eqref{eq:bias_sin2} shows that the actual error depends sensitively on the value of $\chi_P(m/P)$. For instance, if $m/P$ is near an integer $n$, $\chi_P(m/P)$ is close to 1, and the bias $B_P[f]$ approaches $-1/2+C(m)$. The magnitude $|B_P[f]|$ can thus be close to $1/2$, potentially much larger than the $O(P^{-2})$ bound suggests, especially for moderate $P$ or large $m$. Conversely, if $m/P$ is near a cancellation point $n/P$ ($n \not\equiv 0 \pmod{P}$), $\chi_P(m/P)$ is close to zero, leading to a bias $B_P[f] \approx C(m)$, which could be much smaller than the bound predicts, especially if $m$ is large making $C(m)$ small.

The behavior in extreme parameter regimes further illustrates the RBF's precision. As $m \to 0$, we have $C(m) \to 1/2$ and $\chi_P(m/P) \to \chi_P(0)=1$. Substituting into \eqref{eq:bias_sin2}, $B_P[f] \to -1/2+1/2=0$, correctly reflecting the zero error for integrating a constant function (since $f(x) \to 0$ as $m \to 0$). When $m/P$ is near an integer $n$, where $\chi_P(m/P) \approx 1$, the bias magnitude $|B_P[f]| \approx |-1/2+C(m)|$. This sensitivity to near-resonant rational alignments ($m/P \approx n$) underscores the importance of considering grid-frequency interactions in applications involving high frequencies or requiring high precision, such as long-term dynamical simulations or signal processing. The RBF provides the framework for analysing these structurally determined error patterns.

The validity of Theorem \ref{thm:bias_sin2} and its ability to capture the detailed structure not captured by asymptotic bounds is confirmed numerically in Figure \ref{fig:bias_validation}. This figure shows perfect agreement between the directly computed error $|I_P - I|$ and the RBF prediction $|B_P[f]|$ from the theorem across a range of $P$ values. Notably, the RBF correctly determines details such as the error plateau near $P=92$, where the cancellation condition $m/P=1/20$ is met, demonstrating the framework's analytical power beyond asymptotic bounds.

\begin{figure}[htbp]
\centering
\includegraphics[width=0.8\textwidth]{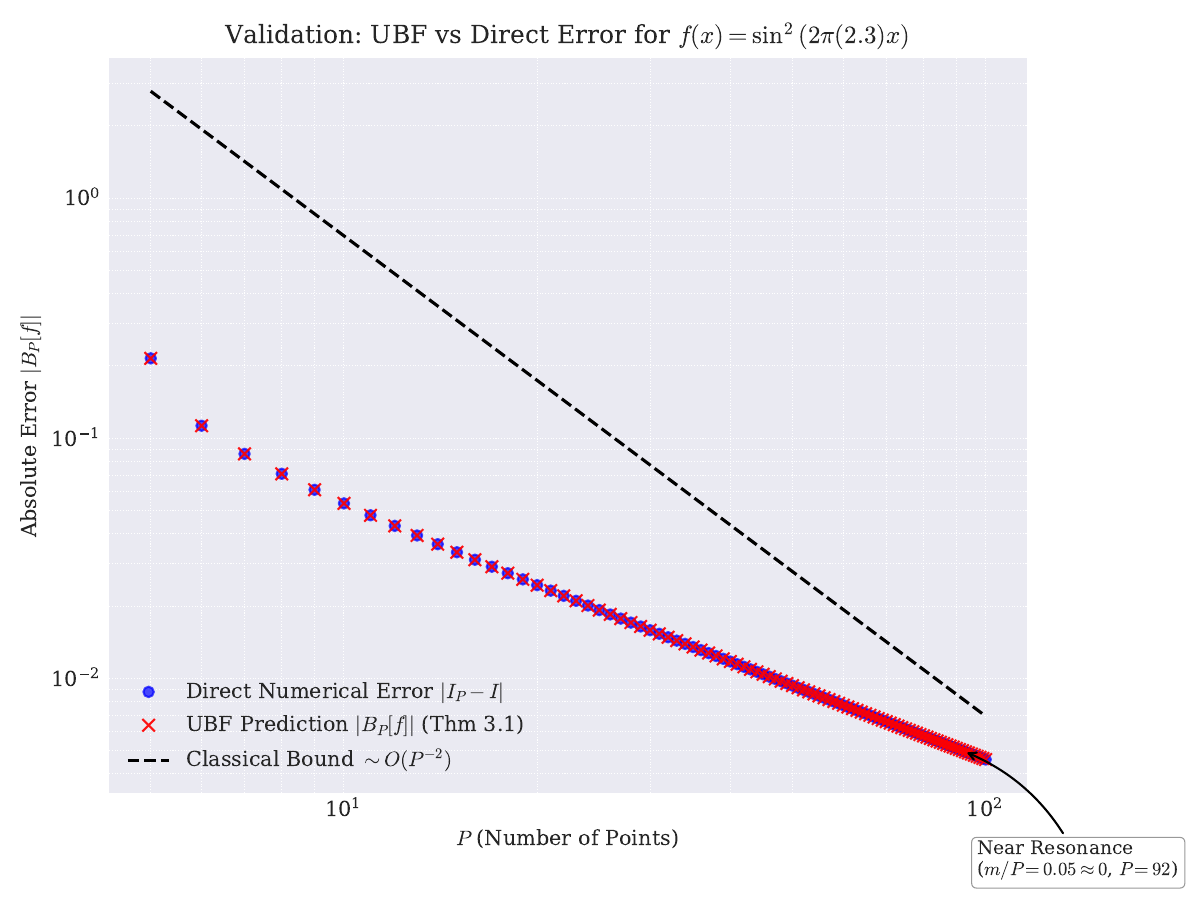} 
\caption{Validation: RBF vs Direct Error for $f(x) = \sin^2(2\pi(2.3)x)$. Log-log plot of the absolute error $|B_P[f]|$ versus the number of points $P$. The Direct Numerical Error $|I_P - I|$ (blue dots) is perfectly matched by the RBF Prediction $|B_P[f]|$ from Theorem \ref{thm:bias_sin2} (red crosses), confirming the formula. Both oscillate around the classical $O(P^{-2})$ bound (dashed line). The annotation highlights the cancellation case at $P=92$, where $m/P = 4.6/92 = 1/20$, resulting in $\chi_{92}(1/20)=0$ and the error plateauing at $B_{92}[f] \approx C(m)$.}
\label{fig:bias_validation}
\end{figure}

\section{The Resonance Function \texorpdfstring{$\tilde{\chi}_P(y)$}{χ̃\_P(y)}: Properties and Interpretation}

Section 3 established that the bias of the trapezoidal rule for the prototype function $f(x)=\sin^2(2\pi kx)$ is governed by the real resonance function $\chi_P(y)$, evaluated at the relative frequency $y=m/P$ (where $m=2k$). This section delves into the mathematical structure and interpretation of $\chi_P(y)$ and its complex counterpart $\tilde{\chi}_P(y)$. These functions are pivotal to the Resonance Bias Framework (RBF), as they precisely quantify the interaction between a function's frequency content and the discrete sampling grid. A thorough understanding of their properties illuminates \textit{why} the quadrature error exhibits characteristic resonance patterns and provides the essential foundation for generalizing the RBF beyond the initial prototype function in Section 5. We demonstrate that $\tilde{\chi}_P(y)$ acts as a structured spectral filter, forging a clear link between the RBF perspective and classical concepts of aliasing and harmonic analysis on finite abelian groups.

\paragraph{4.1 Formal Definition and Role in the RBF}
We analyse the resonance functions for grid size $P \in \mathbb{Z}^+, P \ge 2$, consistent with the trapezoidal rule context of Section 3. (The case $P=1$ is trivial, yielding $\tilde{\chi}_1(y)=e^{2\pi iy \cdot 0}=1$ for all $y$).

\begin{definition}[Resonance Functions] \label{def:resonance_functions}
For $P \in \mathbb{Z}^+, P \ge 1$ and $y \in \mathbb{R}$, the complex resonance function $\tilde{\chi}_P(y)$ and the real resonance function $\chi_P(y)$ are defined as:
\begin{equation} \label{eq:def_chi_tilde}
\tilde{\chi}_P(y) := \frac{1}{P} \sum_{j=0}^{P-1} e^{2\pi i y j}
\end{equation}
\begin{equation} \label{eq:def_chi_real}
\chi_P(y) := \text{Re}[\tilde{\chi}_P(y)] = \frac{1}{P} \sum_{j=0}^{P-1} \cos(2\pi y j)
\end{equation}
\end{definition}

Theorem \ref{thm:bias_sin2} revealed $\chi_P(m/P)$ as the key determinant of the bias for the $\sin^2$ prototype. The complex function $\tilde{\chi}_P(y)$ generalises this role to handle complex Fourier coefficients of arbitrary periodic functions (Section 5). It effectively measures the grid's response to a frequency $y$, acting as a filter that reveals how input frequencies alias onto the grid's resonant modes ($k=lP$). This provides structural insight into the aliasing mechanism described by classical error formulas, such as \(B_P = \sum_{l \neq 0} c_{lP}\) \eqref{eq:classical_error} \cite{Trefethen2013}, by explicitly characterizing the grid's frequency-selective behavior.

\paragraph{4.2 Fundamental Properties}
The analytical behavior of the resonance functions stems from several fundamental properties:

\begin{proposition}[Properties of $\mathbf{\tilde{\chi}_P}$ and $\mathbf{\chi_P}$] \label{prop:chi_properties}
For $P \ge 2$, the functions $\tilde{\chi}_P(y)$ and $\chi_P(y)$ satisfy:
\begin{enumerate}
    \item \textbf{Periodicity and Integer Values (Resonance):} $\tilde{\chi}_P(y+1)=\tilde{\chi}_P(y)$, $\chi_P(y+1)=\chi_P(y)$. Furthermore, for any integer $k \in \mathbb{Z}$, $\tilde{\chi}_P(k)=1$ and $\chi_P(k)=1$.
    \begin{proof}[Proof Sketch] 
    Periodicity follows from $e^{2\pi i (y+1) j} = e^{2\pi i y j}$. For $y=k \in \mathbb{Z}$, $e^{2\pi i k j} = 1$, so $\tilde{\chi}_P(k) = \frac{1}{P} \sum 1 = 1$.
    \end{proof}

    \item \textbf{Conjugate Symmetry and Parity:} $\overline{\tilde{\chi}_P(y)} = \tilde{\chi}_P(-y)$, which implies $\chi_P(y)$ is an even function, $\chi_P(-y)=\chi_P(y)$.
    \begin{proof}[Proof Sketch]
    $\tilde{\chi}_P(-y) = \frac{1}{P} \sum e^{-2\pi i y j} = \overline{\tilde{\chi}_P(y)}$. Taking real parts gives $\chi_P(-y)=\chi_P(y)$.
    \end{proof}

    \item \textbf{Boundedness:} $|\tilde{\chi}_P(y)| \le 1$ and $|\chi_P(y)| \le 1$. Equality $|\tilde{\chi}_P(y)|=1$ holds if and only if $y \in \mathbb{Z}$.
    \begin{proof}[Proof Sketch]
    Triangle inequality: $|\tilde{\chi}_P(y)| \le \frac{1}{P} \sum |e^{2\pi i y j}| = 1$. Equality requires alignment, i.e., $y \in \mathbb{Z}$.
    \end{proof}

    \item \textbf{Rational Zeros (Cancellation):} If $y=n/P$ where $n \in \mathbb{Z}$ and $n \not\equiv 0 \pmod{P}$, then $\tilde{\chi}_P(n/P)=0$ and $\chi_P(n/P)=0$.
    \begin{proof}[Proof Sketch]
    Let $\zeta = e^{2\pi i n/P} \neq 1$. The geometric sum $\sum_{j=0}^{P-1} \zeta^j = \frac{1-\zeta^P}{1-\zeta}$. Since $\zeta^P = e^{2\pi i n} = 1$, the sum is zero. Thus $\tilde{\chi}_P(n/P)=0$.
    \end{proof}

    \item \textbf{Continuity:} $\tilde{\chi}_P(y)$ and $\chi_P(y)$ are continuous functions for all $y \in \mathbb{R}$.
    \begin{proof}[Proof Sketch]
    Finite sums of continuous functions.
    \end{proof}
\end{enumerate}
\end{proposition}

These properties collectively define $\tilde{\chi}_P(y)$ as a continuous, periodic function mapping $\mathbb{R}$ to the closed unit disk in $\mathbb{C}$. Its behavior is characterised by sharp peaks of magnitude 1 at integer arguments (resonances) and exact zeros at non-integer rational arguments with denominator $P$ (cancellations).

\paragraph{4.3 Closed Form and Connection to Dirichlet Kernel}
For $y \notin \mathbb{Z}$, the geometric sum defining $\tilde{\chi}_P(y)$ yields an explicit closed form. With $r=e^{2\pi i y} \neq 1$:
\[
\tilde{\chi}_P(y) = \frac{1}{P} \sum_{j=0}^{P-1} r^j = \frac{1}{P} \frac{1 - r^P}{1 - r} = \frac{1}{P} \frac{1 - e^{2\pi i P y}}{1 - e^{2\pi i y}}
\]
Using the identity $1 - e^{i\theta} = e^{i\theta/2}(-2i \sin(\theta/2))$, we obtain the standard representation:
\begin{equation} \label{eq:chi_tilde_closed_form}
\tilde{\chi}_P(y) = \frac{1}{P} \frac{e^{i\pi P y}(-2i \sin(\pi P y))}{e^{i\pi y}(-2i \sin(\pi y))} = \frac{1}{P} e^{i\pi (P-1) y} \frac{\sin(\pi P y)}{\sin(\pi y)}
\end{equation}
Taking the real part provides the closed form for $\chi_P(y)$:
\begin{equation} \label{eq:chi_real_closed_form}
\chi_P(y) = \text{Re}[\tilde{\chi}_P(y)] = \frac{1}{P} \frac{\sin(\pi P y)}{\sin(\pi y)} \cos(\pi (P-1) y)
\end{equation}
The continuity at integer values $y=k$ is confirmed by applying L'Hôpital's rule to \eqref{eq:chi_tilde_closed_form}, yielding $\lim_{y \to k} \tilde{\chi}_P(y) = 1$.

The sum \(P \tilde{\chi}_P(y) = \sum_{j=0}^{P-1} e^{2\pi i y j}\) is closely related to the standard Dirichlet kernel \(D_{N}(\theta) = \sum_{n=-N}^{N} e^{in\theta}\) used in Fourier series convergence \cite{Katznelson2004}. Specifically, \(P \tilde{\chi}_P(y) = e^{i \pi (P-1) y} \frac{\sin(\pi P y)}{\sin(\pi y)}\), which involves the same ratio of sines as \(D_{P-1}(2\pi y)\) but with a different phase factor due to the unilateral summation. Both embody the principles of phase interference characterizing discrete sampling.

\paragraph{4.4 Geometric Interpretation: Rotating Arrows Model}
The definition \eqref{eq:def_chi_tilde} admits a clear geometric interpretation, visualised in Figure \ref{fig:arrows_geometric}. Each term $z_j = e^{2\pi i y j}$ represents a unit vector (an "arrow") in the complex plane. The sequence $z_0, z_1, \dots, z_{P-1}$ starts at $z_0=1$ and proceeds by successive rotations, each by an angle of $2\pi y$ radians. The complex resonance function $\tilde{\chi}_P(y)$ is the vector average (centroid) of the endpoints of these $P$ arrows. This model provides intuition for the key properties (illustrated in Figure \ref{fig:arrows_geometric}):
\begin{itemize}
    \item \textbf{Full Resonance ($\mathbf{y \in \mathbb{Z}}$):} Rotation angle is a multiple of $2\pi$. All vectors align at $1+0i$. Average is $\tilde{\chi}_P(k)=1$.
    \item \textbf{Cancellation ($\mathbf{y=n/P, n \not\equiv 0 \pmod{P}}$):} Vectors form the $P$-th roots of unity (possibly repeated), distributed symmetrically. Vector sum is zero, $\tilde{\chi}_P(n/P)=0$.
    \item \textbf{Intermediate Cases:} Partial alignment and cancellation. Average $\tilde{\chi}_P(y)$ lies strictly inside the unit disk ($|\tilde{\chi}_P(y)|<1$).
\end{itemize}
This geometric viewpoint visualises the resonance landscape dictated by the arithmetic nature of $y$.

\begin{figure}[htbp]
\centering
\includegraphics[width=\textwidth]{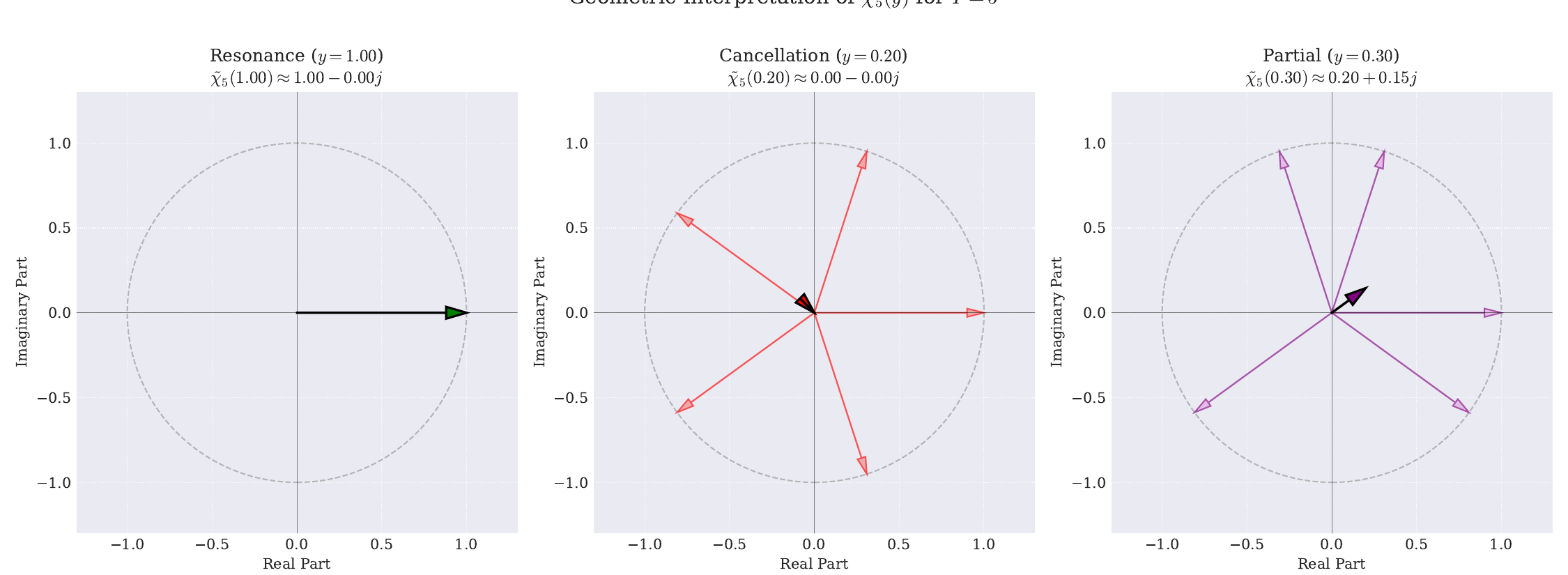} 
\caption{Geometric interpretation of the complex resonance function $\tilde{\chi}_P(y)$ for $P=5$. Each plot shows $P=5$ unit vectors (arrows) starting from the origin, representing $e^{2\pi i y j}$ for $j=0,...,4$. The thick black arrow represents the sum vector scaled by $1/P$, which is $\tilde{\chi}_P(y)$. (Left) Resonance ($y=1.00$): All vectors align, $\tilde{\chi}_5(1) = 1$. (Center) Cancellation ($y=0.20 = 1/5$): Vectors form roots of unity, sum is zero, $\tilde{\chi}_5(0.20) = 0$. (Right) Partial Cancellation ($y=0.30$): Vectors partially cancel, $\tilde{\chi}_5(0.30) \approx 0.20 + 0.15j$.}
\label{fig:arrows_geometric}
\end{figure}

\paragraph{4.5 Connection to Discrete Fourier Transform and $\mathbf{\mathbb{Z}/P\mathbb{Z}}$}
The resonance function is intrinsically linked to harmonic analysis on the finite cyclic group $\mathbb{Z}/P\mathbb{Z}$. The additive characters of this group are $\psi_k(j) = e^{2\pi i k j / P}$ for $k=0, \dots, P-1$, forming an orthonormal basis \cite{SteinShakarchi2003}.
The Discrete Fourier Transform (DFT) uses these characters. For a sequence $g = (g_0, \dots, g_{P-1})$, the $k$-th DFT coefficient is $\hat{g}_k = \sum_{j=0}^{P-1} g_j \overline{\psi_k(j)}$. Considering the constant sequence $g_j = 1$, its $k$-th DFT coefficient (up to conjugation/normalisation) involves our sum:
\[
\sum_{j=0}^{P-1} 1 \cdot \psi_k(j) = \sum_{j=0}^{P-1} e^{2\pi i k j / P} = P \cdot \tilde{\chi}_P(k/P)
\]
The fundamental property $\tilde{\chi}_P(k/P) = \delta_{k \equiv 0 \pmod{P}}$ (from Proposition \ref{prop:chi_properties}), is a direct statement of the character orthogonality relation $\sum_{j=0}^{P-1} \psi_k(j) = P \delta_{k0 \pmod P}$. This connection clarifies the mechanism behind classical error formulas like \(B_P = \sum_{l \neq 0} c_{lP}\) \eqref{eq:classical_error}: the grid sampling, via the filter $\tilde{\chi}_P(k/P)$, selects only frequency components $k=lP$. The RBF provides the explicit filter $\tilde{\chi}_P$ responsible.

\paragraph{4.6 Fine Structure, Convergence, and Number-Theoretic Connections}
The resonance function exhibits additional structure, as detailed in Figure \ref{fig:chi_P_detail}.
\begin{itemize}
    \item \textbf{Near Resonance:} Taylor expansion of $\chi_P(y)$ around an integer $k$ shows quadratic decay. Using the second derivative $\chi_P''(k) = -\frac{2\pi^2 (P-1)(2P-1)}{3}$ (derived from the definition \eqref{eq:def_chi_real}), the expansion $\chi_P(k+\epsilon) = \chi_P(k) + \epsilon \chi_P'(k) + \frac{\epsilon^2}{2} \chi_P''(k) + O(\epsilon^3)$ becomes:
    \begin{equation} \label{eq:taylor_near_resonance}
    \chi_P(k + \epsilon) = 1 - \frac{\pi^2 (P-1)(2P-1)}{3} \epsilon^2 + O(\epsilon^3)
    \end{equation}
    This defines the shape of the main resonance lobe, confirming the peak is a local maximum.

    \item \textbf{Near Zeros:} Near $y=n/P$ ($n \not\equiv 0 \pmod{P}$), $\chi_P(y)$ oscillates rapidly (see Figure \ref{fig:chi_P_detail}), governed by the $\sin(\pi P y)$ numerator in \eqref{eq:chi_real_closed_form}.

    \item \textbf{Zeros and Farey Sequences:} The set of primitive zeros $\{n/P \mid \gcd(n,P)=1\}$ (marked by crosses in Figure \ref{fig:chi_P_detail}) relates to the structure of rational numbers. The distribution of all zeros $n/P$ as $P$ varies mirrors properties of Farey sequences \cite{HardyWright1979}, which order rationals by denominator size, reflecting grid structure and coprimality.

    \item \textbf{Convergence Properties:}
    \begin{remark} 
    As $P \to \infty$:
    \begin{itemize}
        \item \textbf{Distributional Limit:} $P \chi_P(y) \to \sum_{k \in \mathbb{Z}} \delta(y-k)$ (Dirac comb) \cite{Strichartz1994}.
        \item \textbf{Pointwise Limit (Irrational $\mathbf{y}$):} $\lim_{P \to \infty} \tilde{\chi}_P(y) = 0$ by equidistribution (Weyl's criterion) \cite{KuipersNiederreiter1974}.
        \item \textbf{No Pointwise Limit (Rational $\mathbf{y}$):} For $y=n/q$ (lowest terms, $q>1$), $\lim_{P \to \infty} \tilde{\chi}_P(y)$ typically does not exist (depends on $P \pmod q$).
    \end{itemize}
    These highlight the complex interplay between frequency type and grid resolution.
    \end{remark}
\end{itemize}

\begin{figure}[htbp]
\centering
\includegraphics[width=0.8\textwidth]{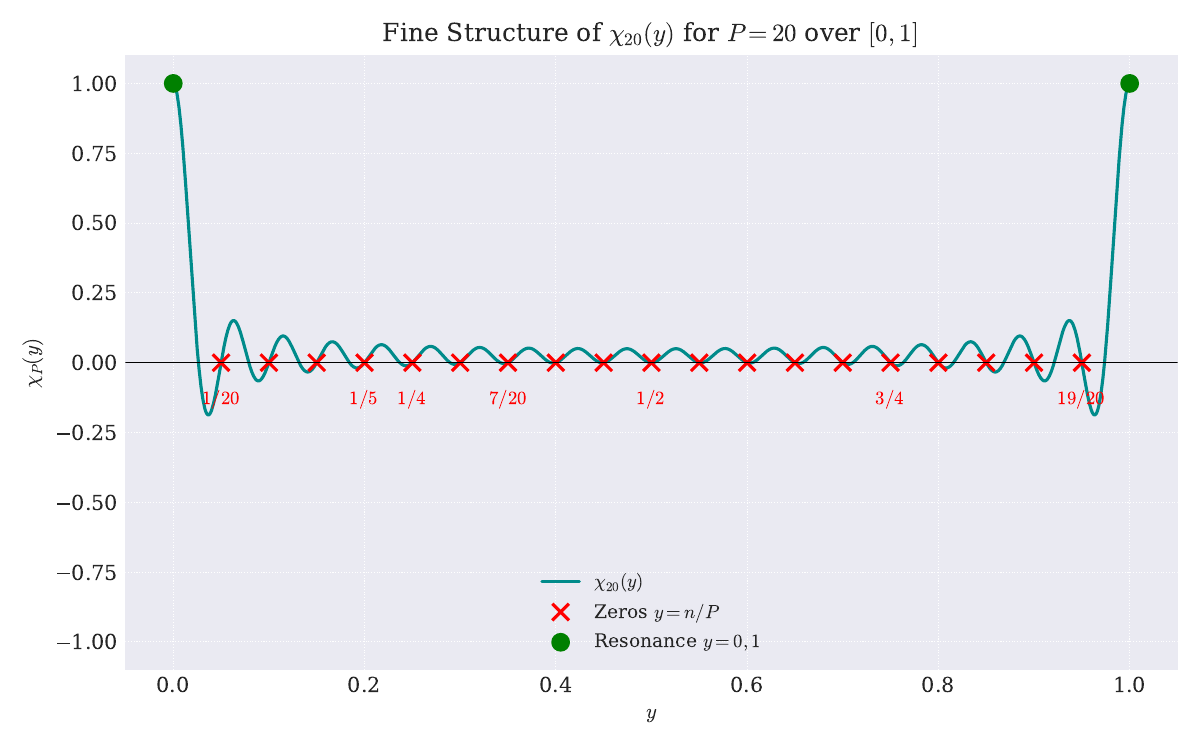} 
\caption{Fine Structure of the real resonance function $\chi_P(y)$ for $P=20$ over the interval $[0, 1]$. The plot shows the function $\chi_{20}(y)$. Resonance peaks ($\chi_P=1$) occur at the endpoints $y=0$ and $y=1$ (green circles). Exact zeros ($\chi_P=0$) occur at $y=n/P$ for $n=1, \dots, 19$ (red crosses). Key rational zeros are labeled (e.g., 1/20, 1/5, 1/4, 7/20, 1/2, 3/4, 19/20), illustrating the cancellation property.}
\label{fig:chi_P_detail}
\end{figure}

\section{Unification with Fourier Error Analysis: The General RBF Bias Formula}
Sections 3 and 4 established the Resonance Bias Framework (RBF) for the prototype function $f(x)=\sin^2(2\pi kx)$, demonstrating the bias is governed by the resonance function $\chi_P(m/P)$. This section generalises these findings to arbitrary smooth, 1-periodic functions. We derive a general bias formula within the RBF employing the complex resonance function $\tilde{\chi}_P(y)$. Crucially, we show this RBF formulation is mathematically equivalent to the classical aliasing sum \eqref{eq:classical_error}. However, the RBF perspective offers structural insight by framing the error as a filtering process, where the grid's response function $\tilde{\chi}_P(y)$ modulates the function's spectrum.

\paragraph{5.1 General Setup for Periodic Functions}
We consider a complex-valued, 1-periodic function $f(x)$ defined on $\mathbb{R}$. We assume that its Fourier series
\[
f(x) = \sum_{k=-\infty}^{\infty} c_k e^{2\pi i k x}, \quad \text{where } c_k = \int_0^1 f(x) e^{-2\pi i k x} dx,
\]
converges absolutely (i.e., $\sum_{k} |c_k| < \infty$). This assumption is crucial as it justifies the interchange of summation orders required in the subsequent derivation of the bias formula. Absolute convergence is guaranteed if, for instance, $f \in C^1(\mathbb{T})$ (the space of continuously differentiable 1-periodic functions on the torus $\mathbb{T} = \mathbb{R}/\mathbb{Z}$). Greater smoothness (e.g., $f \in C^s(\mathbb{T})$ or analytic) leads to faster decay of the coefficients $c_k$, ensuring faster convergence of the quadrature error, as discussed further in Section 5.5. The exact integral of $f$ over one period is simply the zeroth Fourier coefficient: $I[f] = \int_0^1 f(x) dx = c_0$.

\paragraph{5.2 Derivation of the General RBF Bias Formula}
The $P$-point composite trapezoidal rule approximates the integral $I[f]=c_0$ using the average function value over the grid $x_j=j/P$:
\begin{equation}
I_P[f] = \frac{1}{P} \sum_{j=0}^{P-1} f\left(\frac{j}{P}\right), \quad (P \ge 2).
\end{equation}
Substituting the Fourier series representation of $f(x)$ yields:
\[
I_P[f] = \frac{1}{P} \sum_{j=0}^{P-1} \left( \sum_{k=-\infty}^{\infty} c_k e^{2\pi i k (j/P)} \right).
\]
Leveraging the absolute convergence established in Section 5.1 (i.e., $\sum |c_k| < \infty$), we can interchange the order of summation:
\[ 
I_P[f] = \sum_{k=-\infty}^{\infty} c_k \left( \frac{1}{P} \sum_{j=0}^{P-1} e^{2\pi i (k/P) j} \right).
\] 
Recognizing the inner term as the complex resonance function $\tilde{\chi}_P(y)$ evaluated at $y=k/P$ (Definition \ref{def:resonance_functions}):
\[
\frac{1}{P} \sum_{j=0}^{P-1} e^{2\pi i (k/P) j} = \tilde{\chi}_P\left(\frac{k}{P}\right),
\]
we obtain the discrete approximation expressed as the function's spectrum filtered by the resonance function:
\begin{equation} \label{eq:Ip_filtered_spectrum}
I_P[f] = \sum_{k=-\infty}^{\infty} c_k \tilde{\chi}_P\left(\frac{k}{P}\right).
\end{equation}
The bias is the difference between the approximation and the exact integral: $B_P[f] = I_P[f] - I[f] = I_P[f] - c_0$. Since $\tilde{\chi}_P(0/P) = \tilde{\chi}_P(0) = 1$ (Proposition \ref{prop:chi_properties}), the $k=0$ term in the sum \eqref{eq:Ip_filtered_spectrum} is $c_0 \tilde{\chi}_P(0/P) = c_0$. Subtracting this term isolates the bias contributed by the non-zero frequencies:

\begin{theorem}[General RBF Bias Formula] \label{thm:rbf_bias_general}
Let $f(x)$ be a 1-periodic function whose Fourier series $f(x) = \sum_{k \in \mathbb{Z}} c_k e^{2\pi i k x}$ converges absolutely. The bias of the $P$-point composite trapezoidal rule ($P \ge 2$) is given by the analytical expression:\footnote{While this formula is mathematically exact, its practical computation requires knowledge or approximation of the infinite sequence of Fourier coefficients $c_k$.}
\begin{equation} \label{eq:rbf_bias_general}
B_P[f] = \sum_{k \in \mathbb{Z} \setminus \{0\}} c_k \tilde{\chi}_P\left(\frac{k}{P}\right)
\end{equation}
where $\tilde{\chi}_P(y)$ is the complex resonance function defined in Definition \ref{def:resonance_functions}.
\end{theorem}

\subsection{Equivalence to Classical Fourier Error Formula}
We now demonstrate that the RBF bias formula \eqref{eq:rbf_bias_general} is mathematically equivalent to the classical aliasing sum \eqref{eq:classical_error}. This equivalence rests upon the fundamental selective property of the resonance function $\tilde{\chi}_P(y)$ when its argument $y=k/P$ corresponds to integer frequencies $k$.

From Proposition \ref{prop:chi_properties} (specifically Properties 1 and 4), the resonance function acts as an indicator function for divisibility by $P$:
\[
\tilde{\chi}_P\left(\frac{k}{P}\right) =
\begin{cases}
1, & \text{if } k \equiv 0 \pmod{P} \\
0, & \text{if } k \not\equiv 0 \pmod{P}
\end{cases}
\quad \text{for } k \in \mathbb{Z}.
\]
Substituting this property into the general RBF bias formula \eqref{eq:rbf_bias_general}:
\[
B_P[f] = \sum_{k \in \mathbb{Z} \setminus \{0\}} c_k \tilde{\chi}_P\left(\frac{k}{P}\right) = \sum_{\substack{k \in \mathbb{Z} \setminus \{0\} \\ k \equiv 0 \pmod{P}}} c_k \cdot (1) + \sum_{\substack{k \in \mathbb{Z} \setminus \{0\} \\ k \not\equiv 0 \pmod{P}}} c_k \cdot (0).
\]
The second summation vanishes identically. The first summation includes only those non-zero integer indices $k$ that are multiples of $P$. Letting $k=lP$ for $l \in \mathbb{Z} \setminus \{0\}$, the bias simplifies to:
\begin{equation} \label{eq:bias_classical_recovered}
B_P[f] = \sum_{l \in \mathbb{Z} \setminus \{0\}} c_{lP}.
\end{equation}
This is precisely the classical aliasing formula \eqref{eq:classical_error} for the trapezoidal rule error \cite{Trefethen2013}. The derivation highlights that the RBF formulation inherently captures aliasing through the filtering action of $\tilde{\chi}_P(k/P)$. This factor acts effectively as a sampling mask or discrete delta function in the frequency domain, selecting only those modes $k$ that are integer multiples of $P$ to contribute to the bias.

\subsection{Structural Insight from the RBF Formulation}
While mathematically equivalent to the classical sum \eqref{eq:bias_classical_recovered}, the RBF bias formula $B_P[f] = \sum_{k\neq 0} c_k \tilde{\chi}_P(k/P)$ provides enhanced structural interpretation. It explicitly models the error as arising from a filtering process, depicted conceptually in Figure \ref{fig:filter_effect}:
\begin{enumerate}
    \item \textbf{Input Spectrum:} The function's Fourier coefficients $c_k$ represent its intrinsic frequency content.
    \item \textbf{Grid Response Filter:} The complex resonance function $\tilde{\chi}_P(k/P)$ characterises the discrete grid's frequency-selective response at each mode $k$.
    \item \textbf{Output Bias:} The total bias $B_P[f]$ is the sum of the original spectral components $c_k$ ($k \neq 0$), each modulated (weighted) by the grid filter's response $\tilde{\chi}_P(k/P)$ at that frequency.
\end{enumerate}
This filtering perspective directly connects the error structure to the properties of the resonance function $\tilde{\chi}_P(y)$ explored in Section 4.

\begin{figure}[htbp] 
\centering
\includegraphics[width=0.75\textwidth]{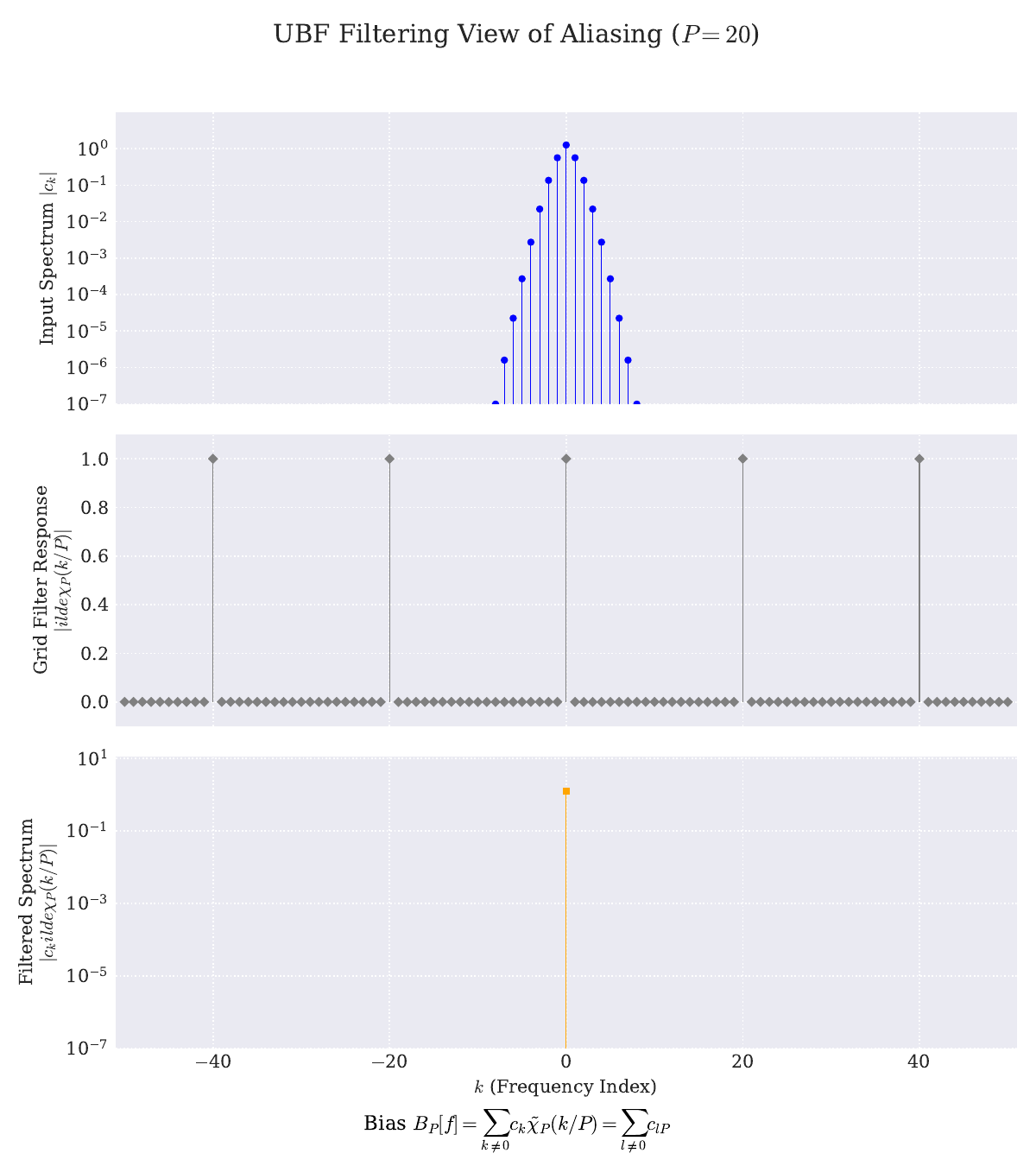} 
\caption{The RBF Filtering View of Aliasing for $P=20$. (Top) Example Input Spectrum $|c_k|$ (log scale, for $f(x)=e^{\cos(2\pi x)}$). (Middle) Grid Filter Response magnitude $|\tilde{\chi}_P(k/P)|$. For integer $k$, this response is $1$ only when $k$ is a multiple of $P=20$ and $0$ otherwise. The filter $\tilde{\chi}_P(k/P)$ thus exhibits a sharp selectivity, isolating the resonant frequencies $k=lP$. (Bottom) Filtered Spectrum magnitude $|c_k \tilde{\chi}_P(k/P)|$ (log scale), showing only the aliased spectral components ($k=lP$, $l \neq 0$) that survive the filtering process. The sum of these filtered components (correctly accounting for phase, not just magnitude) constitutes the bias $B_P[f] = \sum_{k \neq 0} c_k \tilde{\chi}_P(k/P) = \sum_{l \neq 0} c_{lP}$.}
\label{fig:filter_effect}
\end{figure}

This filtering perspective connects the error structure directly to the properties of $\tilde{\chi}_P(y)$ established in Section 4:
\begin{itemize}
    \item \textbf{Resonant (Aliased) Frequencies ($\mathbf{k=lP, l \neq 0}$):} For these frequencies, the relative frequency $y=k/P=l$ is a non-zero integer. Proposition \ref{prop:chi_properties} gives $\tilde{\chi}_P(l)=1$. The filter passes these spectral components $c_{lP}$ with unit gain, recovering the classical aliasing phenomenon.
    \item \textbf{Non-Resonant Integer Frequencies ($\mathbf{k \not\equiv 0 \pmod{P}}$):} For integer frequencies $k$ that are not multiples of $P$, the relative frequency $y=k/P$ is a non-integer rational. Proposition \ref{prop:chi_properties} states $\tilde{\chi}_P(k/P)=0$. The filter perfectly attenuates these components, ensuring they do not contribute to the bias sum \eqref{eq:rbf_bias_general}.
    \item \textbf{Filter Defined for All Frequencies:} It is important to note that while the RBF bias formula \eqref{eq:rbf_bias_general} ultimately only sums contributions where $\tilde{\chi}_P(k/P) \neq 0$ (i.e., for $k=lP$), the filter function $\tilde{\chi}_P(y)$ itself is defined for all real arguments $y$. Its rich structure across all frequencies (including oscillations and behavior at irrational $y$, see Section 4.6) provides a complete picture of the grid's response and may be relevant for analysing related problems or extensions of the framework.
\end{itemize}
Thus, the RBF formulation explicitly reveals how aliasing arises: through the selective filtering action of the grid's inherent resonance function $\tilde{\chi}_P$ acting on the function's spectrum.

This filtering viewpoint provides a causal interpretation: the trapezoidal rule error arises because the discrete grid acts as a frequency filter $\tilde{\chi}_P$, selectively transmitting only the aliased spectral components ($k=lP$) of the integrand $f$. This also clarifies a common point of confusion: simply truncating the Fourier series of $f$ to include only frequencies $|k| < P/2$ (i.e., below the Nyquist frequency relative to the grid spacing) does not guarantee a reduction in quadrature error. The error is fundamentally governed by the coefficients $c_{lP}$, which depend on the spectrum at multiples of $P$, potentially involving high frequencies regardless of the apparent bandwidth below $P/2$.

Furthermore, the RBF formulation naturally accommodates real-valued functions. If $f(x)$ is real, its Fourier coefficients satisfy the Hermitian symmetry $c_{-k} = \overline{c_k}$. Applying this to the bias formula \eqref{eq:bias_classical_recovered},
\[
B_P[f] = \sum_{l \in \mathbb{Z} \setminus \{0\}} c_{lP} = \sum_{l=1}^{\infty} (c_{lP} + c_{-lP}) = \sum_{l=1}^{\infty} (c_{lP} + \overline{c_{lP}}) = \sum_{l=1}^{\infty} 2 \text{Re}(c_{lP}),
\]
which is manifestly real. This confirms the physical consistency of the result derived using complex Fourier analysis and the complex resonance function $\tilde{\chi}_P$.

\subsection{Error Bounds and Convergence Rates}
The RBF framework readily recovers classical error bounds by relating the bias directly to the decay rate of the specific Fourier coefficients $c_{lP}$ selected by the filter $\tilde{\chi}_P(k/P)$.

First, consider functions with finite smoothness. Assume $f \in C^s(\mathbb{T})$ for some integer $s \ge 1$. Standard Fourier analysis results state that this implies $|c_k| = O(|k|^{-s})$ as $|k| \to \infty$. A slightly stronger condition, $f \in C^{s+1}(\mathbb{T})$, implies faster decay, $|c_k| \le C|k|^{-s-1}$ for some constant $C>0$.\footnote{Using the $O(|k|^{-s-1})$ decay (valid for $f \in C^{s+1}$) simplifies the summation argument via the Riemann zeta function. However, the fundamental convergence rate $O(P^{-s})$ associated with $f \in C^s$ can also be derived from the $O(|k|^{-s})$ decay, though the sum requires slightly more care \cite{DavisRabinowitz1984}.} Using this stronger decay and the bias formula $B_P[f] = \sum_{l \neq 0} c_{lP}$, we bound the error magnitude:
\[
|B_P[f]| \le \sum_{l \in \mathbb{Z} \setminus \{0\}} |c_{lP}| \le \sum_{l \neq 0} C |lP|^{-s-1} = C P^{-s-1} \sum_{l \neq 0} |l|^{-s-1}.
\]
The summation term converges to $2\zeta(s+1)$, where $\zeta$ is the Riemann zeta function (convergent for $s+1 \ge 2$, i.e., $s \ge 1$). This gives the bound:
\[
|B_P[f]| \le (2 C \zeta(s+1)) P^{-s-1} = O(P^{-s-1}).
\]
This demonstrates algebraic convergence. (Under the standard $f \in C^s$ condition, the rate is $O(P^{-s})$ \cite{Trefethen2013}).

Next, consider analytic functions on $\mathbb{T}$. For such functions, the Fourier coefficients decay exponentially: $|c_k| \le C e^{-\gamma |k|}$ for some constants $C, \gamma > 0$. Applying this to the bias sum yields:
\[
|B_P[f]| \le \sum_{l \neq 0} C e^{-\gamma |lP|} = 2C \sum_{l=1}^\infty (e^{-\gamma P})^l = 2C \frac{e^{-\gamma P}}{1 - e^{-\gamma P}}.
\]
Since $e^{-\gamma P} \to 0$ as $P \to \infty$, the bound simplifies to $O(e^{-\gamma P})$, matching the well-known exponential (or geometric) convergence rate for analytic periodic functions \cite{Waldvogel2006, Trefethen2013, TrefethenWeideman2014}.

In both cases, the RBF perspective clarifies that these convergence rates arise directly from the decay properties of the specific Fourier coefficients $c_{lP}$ that are selectively passed by the grid's resonance filter $\tilde{\chi}_P(k/P)$.

\subsection{Extension to Two Dimensions} \label{sec:2d_extension}

The RBF framework extends naturally to higher dimensions. We illustrate this for integration over the unit square $[0,1]^2$ using a bivariate function $f(x_1, x_2)$ that is 1-periodic in both variables. The $P \times P$ tensor-product trapezoidal rule uses grid points $\mathbf{x}_{\mathbf{j}} = \mathbf{j}/P$, where $\mathbf{j} = (j_1, j_2)$ with $j_1, j_2 \in \{0, \dots, P-1\}$. The approximation is:
\[
I_P[f] = \frac{1}{P^2} \sum_{j_1=0}^{P-1} \sum_{j_2=0}^{P-1} f\left(\frac{j_1}{P}, \frac{j_2}{P}\right).
\]
Using the 2D Fourier expansion $f(\mathbf{x}) = \sum_{\mathbf{k} \in \mathbb{Z}^2} c_{\mathbf{k}} e^{2\pi i \mathbf{k} \cdot \mathbf{x}}$ and interchanging sums (assuming absolute convergence):
\[
I_P[f] = \sum_{\mathbf{k} \in \mathbb{Z}^2} c_{\mathbf{k}} \left( \frac{1}{P^2} \sum_{j_1=0}^{P-1} \sum_{j_2=0}^{P-1} e^{2\pi i (k_1 j_1 / P + k_2 j_2 / P)} \right).
\]
The double sum defines the 2D complex resonance function. Due to the separability of the exponential term, this function factorises into a product involving the 1D resonance function $\tilde{\chi}_P$ (Definition \ref{def:resonance_functions}):
\begin{align}
\tilde{\chi}_P^{(2)}\left(\frac{k_1}{P}, \frac{k_2}{P}\right) &:= \frac{1}{P^2} \sum_{j_1, j_2=0}^{P-1} e^{2\pi i (k_1 j_1 / P + k_2 j_2 / P)} \nonumber \\
&= \left( \frac{1}{P} \sum_{j_1=0}^{P-1} e^{2\pi i k_1 j_1 / P} \right) \left( \frac{1}{P} \sum_{j_2=0}^{P-1} e^{2\pi i k_2 j_2 / P} \right) \nonumber \\
&= \tilde{\chi}_P\left(\frac{k_1}{P}\right) \tilde{\chi}_P\left(\frac{k_2}{P}\right). \label{eq:2d_resonance_factorisation}
\end{align}
Thus, the 2D resonance filter $\tilde{\chi}_P^{(2)}$ is simply the product of two 1D filters $\tilde{\chi}_P$. The general RBF bias formula in 2D, $B_P[f] = I_P[f] - c_{\mathbf{0}}$, then becomes:
\begin{equation} \label{eq:rbf_bias_2d}
B_P[f] = \sum_{\mathbf{k} \in \mathbb{Z}^2 \setminus \{\mathbf{0}\}} c_{\mathbf{k}} \, \tilde{\chi}_P\left(\frac{k_1}{P}\right) \tilde{\chi}_P\left(\frac{k_2}{P}\right).
\end{equation}
Recalling that $\tilde{\chi}_P(k_i/P)$ is non-zero (and equals 1) if and only if $k_i \equiv 0 \pmod{P}$ (Proposition \ref{prop:chi_properties}), the product filter $\tilde{\chi}_P(k_1/P) \tilde{\chi}_P(k_2/P)$ is non-zero (and equals 1) if and only if both $k_1 = l_1 P$ and $k_2 = l_2 P$ for integers $l_1, l_2$. Therefore, equation \eqref{eq:rbf_bias_2d} simplifies directly to the classical 2D aliasing formula:
\begin{equation} \label{eq:classical_alias_2d}
B_P[f] = \sum_{\mathbf{l} \in \mathbb{Z}^2 \setminus \{\mathbf{0}\}} c_{\mathbf{l}P},
\end{equation}
where $\mathbf{l}P = (l_1 P, l_2 P)$. This equivalence is confirmed numerically for test functions in Appendix A (Table \ref{tab:2d_validation_appendix}), validating the 2D RBF formulation.

The corresponding 2D real resonance landscape, $\chi_P^{(2)}(y_1, y_2) = \mathrm{Re}[\tilde{\chi}_P(y_1) \tilde{\chi}_P(y_2)]$, inherits its structure from the 1D function, exhibiting peaks at integer coordinates $(\mathbb{Z}^2)$ and lines of cancellation when $y_1$ or $y_2$ equals $n/P$ ($n \not\equiv 0 \pmod P$), as shown in Appendix A (Figure \ref{fig:2d_resonance_appendix}).

\section{Discussion}

The Resonance Bias Framework (RBF) provides a precise and structurally informative understanding of trapezoidal rule errors for periodic functions. Departing from traditional analyses focused on asymptotic bounds or viewing error as statistical noise, the RBF interprets the error as a deterministic, finite-sample phenomenon governed by the complex resonance function $\tilde{\chi}_P(y)$.

While mathematically equivalent to the classical aliasing sum $B_P = \sum_{l \neq 0} c_{lP}$, the RBF formulation, expressed as
\begin{equation} \label{eq:rbf_bias_repeat_discussion}
B_P[f] = \sum_{k \ne 0} c_k \, \tilde{\chi}_P\left(\frac{k}{P}\right),
\end{equation}
distinctly reveals the underlying spectral filtering structure. This perspective clarifies that the total bias arises from modulating each input Fourier mode $c_k$ ($k \neq 0$) by the grid's specific frequency response $\tilde{\chi}_P(k/P)$. Consequently, the RBF highlights how the interaction between the grid resolution $P$ and the function's detailed spectral content (specifically, the coefficients $c_{lP}$), rather than just overall smoothness, dictates the quadrature error.

\subsection{Structural Insight Beyond Asymptotic Analysis}
The RBF emphasises the error structure at any finite grid resolution $P$, offering insights beyond traditional asymptotic analysis. While classical bounds like $O(P^{-s})$ or $O(e^{-\gamma P})$ capture average decay rates, they often obscure the intricate, non-monotonic error behavior. In contrast, the RBF reveals, via the resonance function $\tilde{\chi}_P(y)$ (Section 4), that the error's magnitude and sign depend deterministically on the arithmetic relationship $k/P$ between spectral components and the grid. The resonance landscape (Figures \ref{fig:chi_P_landscape} and \ref{fig:chi_P_detail}), with its sharp peaks (resonances) and exact zeros (cancellations), demonstrates that the error is a \textit{structured signal}, not simply decaying noise. This framework precisely explains phenomena such as error plateaus, oscillations, or sharp transitions as $P$ varies (validated in Figure \ref{fig:bias_validation}), which are crucial in applications where $P$ is fixed or computationally constrained. This is distinct from classical error bounds, which are often asymptotic (large $P$) or provide worst-case estimates. The RBF formula describes the error structure for any finite $P$ and captures the specific grid-function interaction.

\subsection{Number-Theoretic Dimensions}
The RBF explicitly connects quadrature error to number-theoretic concepts. The behavior of $\tilde{\chi}_P(y)$ is intrinsically linked to the rational arithmetic of $y = k/P$, featuring properties like exact zeros at $y = n/P$ ($n \not\equiv 0 \pmod P$, see Figure \ref{fig:chi_P_detail}) and high sensitivity near integer resonances. This highlights the role of rational approximation and grid alignment effects. The distribution and structure of these zeros relate to patterns found in Farey sequences and continued fractions \cite{HardyWright1979}, reflecting grid symmetries and coprimality. This connection underscores links to number-theoretic fields like Diophantine approximation and the theory of equidistribution (relevant for the $P \to \infty$ limit, see Section 4.6). These observations imply that quadrature error magnitude depends subtly on the arithmetic nature of frequency indices relative to $P$—a dimension often overlooked by classical approaches focused primarily on function smoothness.

\subsection{Geometric Intuition}
The "rotating arrows" model (Section 4.4, Figure \ref{fig:arrows_geometric}) provides a compelling geometric interpretation for the complex resonance function $\tilde{\chi}_P(y)$. It represents $\tilde{\chi}_P(y)$ as the normalised vector sum (centroid) of $P$ unit vectors rotating incrementally by an angle $2\pi y$. Resonance ($y \in \mathbb{Z}$) corresponds to the perfect constructive alignment of all vectors, while cancellation ($y=n/P, n \not\equiv 0 \pmod P$) arises from vectors forming a symmetric configuration (roots of unity) summing to zero. This visualisation makes the complex interference patterns underlying the resonance landscape intuitive and reinforces $\tilde{\chi}_P(y)$'s role as a phase-sensitive filter kernel governing the grid's response.

\subsection{Potential Applications}
This deterministic, structure-aware perspective on quadrature error suggests several practical implications and potential applications:
\begin{itemize}
    \item \textbf{Diagnosing Error Sources:} Precisely identifying and quantifying error contributions from specific grid-frequency resonances in high-precision calculations, particularly in spectral methods or simulations involving periodic phenomena.
    \item \textbf{Analysing Long-Term Simulations:} Understanding and analysing the potential for structured error accumulation or artificial drift in dynamical systems (e.g., Hamiltonian systems) integrated over long times using fixed time steps.
    \item \textbf{Interpreting Discretisation Artifacts:} Understanding deterministic aliasing patterns in signal processing or fixed-point arithmetic where periodic signals interact with discrete sampling. The resonance function $\tilde{\chi}_P(y)$ (Figure \ref{fig:filter_effect}) directly parallels frequency responses and windowing functions in digital filter design, offering insight into spectral leakage.
    \item \textbf{Informing Algorithm Design (Potential):} Potentially guiding the selection of the number of grid points $P$ based on known spectral features of the integrand to avoid resonances, or inspiring adaptive quadrature methods that dynamically adjust $P$ or shift the grid to mitigate error amplification near resonant conditions.
\end{itemize}

\section{Future Work and Conclusion}

\subsection{Future Research Directions}
The Resonance Bias Framework (RBF) opens several promising avenues for further investigation, spanning theoretical extensions and computational applications.

\textbf{Theoretical Extensions:}
\begin{enumerate}
    \item \textit{Integrals over $\mathbb{R}$ and Analogues:} Adapt RBF concepts to analyse the structure of the Poisson summation error term $\sum_{k \neq 0} \hat{f}(2\pi k/h)$ arising in quadrature over $\mathbb{R}$ \cite{Trefethen2013, TrefethenWeideman2014}. Investigate whether continuous analogues of the resonance function $\tilde{\chi}_P$ can provide structural insight into error contributions from singularities or complex integration contours, potentially connecting to analyses like those in \cite{Waldvogel2006}. Furthermore, explore potential analogies between the RBF filtering mechanism and physical resonance phenomena.
    \item \textit{Higher Dimensions and General Lattices:} Extend the RBF rigorously to multivariate functions on general lattices, moving beyond the tensor-product grids treated in Section \ref{sec:2d_extension}. Analyse resonance phenomena for non-uniform sampling strategies, lattice rules, and quasi-Monte Carlo point sets.
    \item \textit{Deeper Arithmetic Connections:} Systematically explore the interplay between the fine structure of $\chi_P$ (Figure \ref{fig:chi_P_detail}), number-theoretic objects like Farey sequences and continued fractions, Diophantine approximation properties of frequencies, and potential links to modular forms or spectral theory on arithmetic groups.
\end{enumerate}

\textbf{Computational and Applied Directions:}
\begin{enumerate}
    \setcounter{enumi}{3} 
    \item \textit{Algorithm-Specific Error Analysis:} Apply the RBF perspective to diagnose and understand structured errors in practical algorithms involving discrete sums or transforms of periodic data, such as Fast Fourier Transforms (FFTs), spectral methods for PDEs, and N-body simulations.
    \item \textit{Alternative Quadrature Schemes:} Derive and analyse the characteristic resonance functions associated with other quadrature rules (e.g., Simpson's rule, Gaussian quadrature, Clenshaw-Curtis) when applied to periodic integrands, considering the impact of their specific nodes and weights.
    \item \textit{Wavelet and Multiresolution Methods:} Investigate potential connections or analogues between RBF resonance and error structures observed in wavelet-based or multiresolution numerical methods, aiming to understand scale-dependent error behavior.
\end{enumerate}

\subsection{Conclusion}
The Resonance Bias Framework (RBF) reinterprets trapezoidal rule error for periodic functions not as statistical noise, but as a deterministic, structured signal. This signal is governed by the arithmetic and geometric interplay between the function's Fourier spectrum and the sampling grid, mediated precisely by the complex resonance function $\tilde{\chi}_P(y)$. This perspective unifies the classical aliasing formula $B_P = \sum_{l \neq 0} c_{lP}$ with a mechanistic view of the error as a spectral filtering process (Figure \ref{fig:filter_effect}), revealing an underlying harmonic structure within discretisation effects.

The RBF treats quadrature error as a structured, interpretable object. The resonance landscape (Figures \ref{fig:chi_P_landscape}, \ref{fig:chi_P_detail}, validated in Figure \ref{fig:bias_validation}) dictates the error's detailed behavior, including sharp resonances and cancellations. By weaving together insights from numerical analysis, harmonic analysis (Figure \ref{fig:arrows_geometric}), and number theory, the RBF provides a foundation for developing more transparent, robust, and potentially adaptive numerical methods. Ultimately, the RBF framework offers a valuable lens through which to understand that what might appear as numerical noise often possesses a deep, underlying harmonic structure dictated by the interplay between the function's spectrum and the sampling parameter $P$. It provides a framework for uncovering this underlying structure in numerical phenomena across computational science.

\section*{Author's Note}
{\small
I am a second-year undergraduate in economics at the University of Bristol working independently. This paper belongs to a research programme I am developing across information theory, econometrics and mathematical statistics.

This work was produced primarily through AI systems that I directed and orchestrated. The AI generated the mathematical content, proofs and symbolic derivations based on my research questions and guidance. I have no formal mathematical training but am eager to learn through this process of directing AI-powered mathematical exploration.

My contribution involves designing research directions, evaluating and selecting AI outputs, and ensuring the coherence of the overall research agenda. All previously published work that influenced these results has been cited to the best of my knowledge and research capabilities in my current position. The presentation aims to be pedagogically accessible.
}

\appendix
\section{Appendix A: Supporting Details}
This appendix contains supplementary material, including minimal numerical verification of the RBF framework and a conceptual figure for the 2D resonance landscape.

\subsection{Minimal Numerical Verification}

\begin{table}[ht]
\centering
\caption{Numerical Verification of RBF Bias Formulas.}
\label{tab:2d_validation_appendix}
\begin{tabular}{llcl}
\toprule
Test Case & Method & Computed Bias $B_P[f]$ & Notes \\
\midrule
\multirow{2}{*}{\parbox{4.5cm}{\textbf{Case 1: 1D Prototype} \\ $f(x) = \sin^2(2\pi(2.3)x)$ \\ $P=20$}}
    & Direct $(I_P - I)$ & $-2.444718 \times 10^{-2}$ & Matched \\
    & RBF Proto (Eq.~\ref{eq:bias_sin2}) & $-2.444718 \times 10^{-2}$ & (Diff: $\approx 2.1 \times 10^{-17}$) \\
\addlinespace 
\multirow{3}{*}{\parbox{4.5cm}{\textbf{Case 2: 1D Equivalence} \\ $f(x) = \cos(8\pi x)$ \\ $P=4$}}
    & Direct $(I_P - I)$ & $1.0$ & Matched \\
    & RBF General (Eq.~\ref{eq:rbf_bias_general}) & $1.0$ & \\
    & Classical Alias (Eq.~\ref{eq:classical_error}) & $1.0$ & (Diffs: $0.0$) \\
\addlinespace
\multirow{3}{*}{\parbox{4.5cm}{\textbf{Case 3: 2D Equivalence} \\ $f(x_1,x_2) = \cos(8\pi x_1)\cos(8\pi x_2)$ \\ $P=4$}}
    & Direct $(I_P - I)$ & $1.0$ & Matched \\
    & RBF 2D (Eq.~\ref{eq:rbf_bias_2d}) & $1.0$ & \\
    & Classical 2D Alias (Eq.~\ref{eq:classical_alias_2d}) & $1.0$ & (Diffs: $0.0$) \\
\bottomrule
\end{tabular}
\par 
\textit{Note:} Calculations performed using Python/NumPy based on the analytical forms of the bias formulas. Differences reflect standard floating-point precision limits. The chosen functions have finite Fourier series, making sums exact. Case 1 validates the prototype formula \eqref{eq:bias_sin2}. Cases 2 and 3 validate the general RBF formulas (\eqref{eq:rbf_bias_general}, \eqref{eq:rbf_bias_2d}) and their equivalence to classical aliasing (\eqref{eq:classical_error}, \eqref{eq:classical_alias_2d}), specifically demonstrating the capture of aliasing effects where frequencies are multiples of $P$.
\end{table}

\subsection{2D Resonance Visualisation}

\begin{figure}[htbp]
\centering
\includegraphics[width=0.65\textwidth]{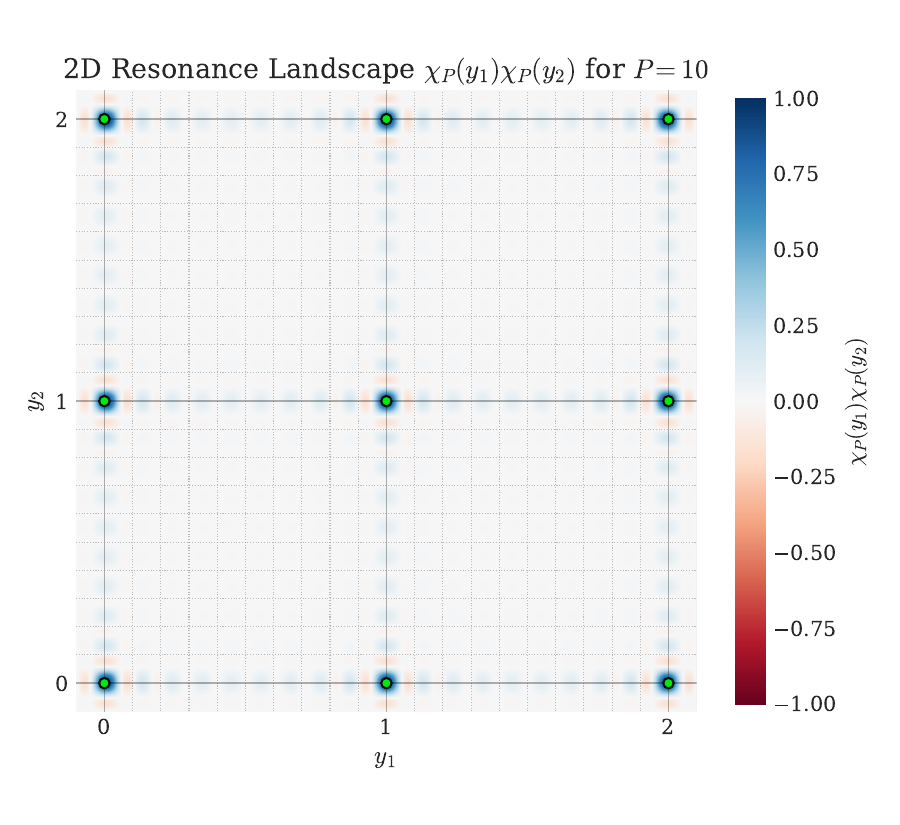} 
\caption{2D Resonance Landscape $\chi_P^{(2)}(y_1, y_2) = \chi_P(y_1)\chi_P(y_2)$ for $P=10$. The heatmap shows the value of the product of two 1D real resonance functions over the domain $[0, 2] \times [0, 2]$. Peaks (value 1, indicated by green circles) occur at integer coordinate pairs $(l_1, l_2)$. The function decays rapidly away from these peaks and exhibits zero values along the grid lines $y_1 = n/10$ and $y_2 = n/10$ for $n \in \{1, ..., 9\}$ (faint dotted grid lines).}
\label{fig:2d_resonance_appendix}
\end{figure}


\bibliographystyle{plain}
\bibliography{references}

\end{document}